\documentclass[a4paper,11pt,titlepage,twoside]{article}

\usepackage{graphicx}
\usepackage[T1]{fontenc} 
\usepackage[utf8]{inputenc}
\usepackage[english]{babel}
\usepackage{soul}
\usepackage{amsfonts}
\usepackage{amsmath}
\usepackage{amsthm}
\usepackage{amssymb}
\usepackage{mathrsfs}
\usepackage[top=5.5cm, bottom=3cm, left=3.8cm, right=3.8cm]{geometry}
\usepackage{setspace}
\usepackage{afterpage}
\usepackage{extarrows}
\usepackage{fancyhdr}
\usepackage{titlesec}
\usepackage{enumitem} \setlist{nosep}
\usepackage[pdftex,breaklinks,colorlinks,linkcolor=blue,
anchorcolor=blue]{hyperref}

\theoremstyle{definition}
\newtheorem{defin}{Definition}[section]
\theoremstyle{plain}
\newtheorem{teor}[defin]{Theorem}
\newtheorem{lem}[defin]{Lemma}
\newtheorem{pro}[defin]{Proposition}
\newtheorem{cor}[defin]{Corollary}
\theoremstyle{definition}
\newtheorem{esm}[defin]{Example}
\newtheorem{osr}[defin]{Remark}
\numberwithin{equation}{section}

\renewcommand{\O}{\Omega}
\newcommand{\D}{\mathcal{D}}
\renewcommand{\H}{\mathcal{H}}

\newcommand{\Set}{\mathcal{S}}
\newcommand{\St}{\mathsf{S}}
\newcommand{\B}{\mathcal{B}}

\newcommand{\Eo}{\mathcal{E}_\Omega}
\newcommand{\rn}{\mathfrak{n}}

\newcommand{\Up}{\Upsilon}
\renewcommand{\L}{\Lambda}
\newcommand{\Po}{\mathfrak{P}}
\newcommand{\n}[1]{\|#1\|}
\newcommand{\nor}{\|\cdot \|}

\newcommand{\noo}{\|\cdot\|_\O}

\renewcommand{\l}{\langle}
\renewcommand{\r}{\rangle}

\newcommand{\R}{\mathbb{R}}
\newcommand{\C}{\mathbb{C}}

\newcommand{\pint}{\l\cdot|\cdot\r}
\newcommand{\pin}[2]{\l#1 | #2\r}
\newcommand{\no}{\noindent}
\newcommand{\ol}{\overline}
\newcommand{\ull}{\underline} 

\newcommand{\Ma}{\mathscr{M}_{\ull{\alpha}}}

\newcommand{\sub}{\subseteq}

\newcommand{\mez}{\frac{1}{2}}
\renewcommand{\ll}{{\it l}}

\titleformat{\section}
{\normalfont\fillast \fontsize{12}{15}\scshape}{\thesection.}{0.8em}{}

\begin{document}
	
\thispagestyle{plain}

\begin{center}
	\large
	{\bf MAXIMAL OPERATORS WITH RESPECT \\ TO THE NUMERICAL RANGE} \\
	\vspace*{0.5cm}
	ROSARIO CORSO
\end{center}

\normalsize 
\vspace*{1cm}	

\small 
	
	\begin{minipage}{11.8cm}
		{\scshape Abstract.} 
		Let $\rn$ be a nonempty, proper, convex subset of $\C$. The $\rn$-maximal operators are defined as the operators having numerical ranges in $\rn$ and are maximal with this property. Typical examples of these are the maximal symmetric (or accretive or dissipative) operators, the associated to some sesquilinear forms (for instance, to closed sectorial forms), and the generators of some strongly continuous semi-groups of bounded operators. In this paper	the $\rn$-maximal operators are studied and some characterizations of these in terms of the resolvent set are given.
	\end{minipage}

	\vspace*{.5cm}
	
	\begin{minipage}{11.8cm}
	\end{minipage}
	
	\vspace*{.5cm}
	
	\begin{minipage}{11.8cm}
	\end{minipage}

\vspace*{1cm}
\normalsize
	
	\section{Introduction}

The {\it numerical range} of an operator $T$ with domain $D(T)$ on a complex Hilbert space $\H$, with inner product $\pint$ and norm $\nor$, is the convex subset of $\C$
$$
\rn_T:=\{\pin{T\xi}{\xi}:\xi\in D(T), \n{\xi}=1\}.
$$
Consider a proper convex subset $\rn$ of $\C$. We say that $T$ is $\rn$-maximal if $\rn_T\subseteq \rn$ and $T$ has no proper extension with this property.
This concept has the maximal symmetric, accretive and dissipative operators as special cases. 

We know by von Neumann \cite{Neumann} that a (densely defined) symmetric operator $T$ is maximal symmetric if and only if either the half-plane  $\{\lambda \in \C:\Im \lambda >0\}$ or the half-plane $\{\lambda \in \C:\Im \lambda <0\}$ is contained in the resolvent set $\rho(T)$. This, in turn, is equivalent to say that a defect index of $T$, $\dim R(T-iI)^\perp$ or $\dim R(T+iI)^\perp$, is zero. Phillips proved a similar result in \cite{Phillips}, i.e., that a densely defined dissipative operator is maximal dissipative if and only if $\lambda \in \rho(T)$ for some $\Re\lambda>0$, if and only if $\{\lambda \in \C: \Re\lambda >0\}\subseteq \rho(T)$.

In this paper we deal with the analogue characterization for a general $\rn$-maximal operator $T$, taking into account that the defect index of $T$ is defined for $\lambda \in \ol{\rn}^c$ (i.e., the complement of the closure of $\rn$) as $\dim R(T-\lambda I)^\perp$, and is constant for all $\lambda$ contained in a connected component of $\ol{\rn}^c$. If $T$ is densely defined and $\rn_T\subseteq \rn$, then $T$ is closable and its closure $\ol{T}$ has numerical range in $\ol{\rn}$. For this reason an assumption that we make, in order to have closed $\rn$-maximal operator, is that $\rn$ is closed.\\
In particular, the new cases studied are given by \hyperref[th_car_n_max_set]{Theorem} \ref{th_car_n_max_set} and \hyperref[th_car_strip]{Theorem} \ref{th_car_strip}, where $\rn$ is a closed subset of a {\it sector} or of a {\it closed strip}, respectively. The first one is based on the Friedrichs extension of a densely defined sectorial operator, while the second one uses the fact that a densely defined operator $T$ with numerical range in a horizontal strip is uniquely decomposable (like for bounded operator) as $T=S+iB$ (\hyperref[lem_decomp_strip]{Lemma} \ref{lem_decomp_strip}), in which $S$ and $B$ are symmetric and $B$ is bounded. If the strip is closed, the sets of extensions of $T$ and $S$ are in a one-to-one correspondence (\hyperref[corr_extens_strip]{Lemma} \ref{corr_extens_strip}).

For many classes of densely defined and $\rn$-maximal operators, the resolvent set contains a connected component of $\ol{\rn}^c$. This property holds for generators of some strongly continuous semi-groups (or groups) of bounded operators, for instance of contractions. We give more examples of semi-groups and corresponding $\rn$-maximal generator in \hyperref[sec:exampl]{Section} \ref{sec:exampl}. We also recall that for an operator $T$ with numerical range in $\rn$ and satisfying $\ol{\rn}^c \subseteq \rho(T)$ it is possible to define a so-called functional calculus developed in many work, for instance in \cite{Batty,Batty2,Delau,Haase_b,McIntosh_fc}.

In \hyperref[sec:corr]{Section} \ref{sec:corr} we talk about correspondences (through a map like Cayley transform) between extensions of an operator with particular numerical range and bounded operators.

Another area where $\rn$-maximal operators occur concerns sesquilinear forms on a Hilbert space. The operator $T$ {\it associated} to a sesquilinear form $\O$, with dense domain $\D$ in $\H$, has domain	
\begin{equation}
\label{eq_dom_intro}
D(T)=\{\xi \in \D:\exists \chi \in \H, \O(\xi,\eta)=\pin{\chi}{\eta}, \forall \eta \in \D\}
\end{equation}
and it is defined by $T\xi=\chi$, for all $\xi\in D(T)$ and $\chi$ as in (\ref{eq_dom_intro}) (the density of $\D$ ensures that this definition is well-posed). Hence, $\O$ is {\it represented} by $T$, i.e.,
\begin{equation}
\label{eq_intro}
\O(\xi,\eta)=\pin{T\xi}{\eta}, \qquad \forall \xi\in D(T),\eta \in \D.
\end{equation}
The domain of $T$ might be very small, therefore one searches conditions so that $T$ is densely defined (and also closed). Representation theorems are studied by Kato (who consider the {\it closed sectorial} forms) and then by several author (see the references of \cite{RC_CT}). Returning to the problem of the maximality,  Kato's result says that the operator associated to a densely defined closed sectorial form is m-sectorial, i.e., $\Set$-maximal, with $\Set$ a sector containing the numerical range of $\O$. In \hyperref[sec:solv]{Section} \ref{sec:solv} we consider {\it solvable} sesquilinear forms, which, with the representation (\ref{eq_intro}), have been defined and studied in \cite{Tp_DB}. The analysis of solvable forms has been continued in \cite{RC_CT,Second}. 
Here, for a general proper, convex subset $\rn$ of $\C$ we give some sufficient conditions for the $\rn$-maximality of the associated operator in \hyperref[th_op_ass_max]{Theorem} \ref{th_op_ass_max} and \hyperref[cor_form_n_max]{Corollary} \ref{cor_form_n_max}. However, if the numerical range of a solvable form is contained in a strip, then the associated operator is always $\rn$-maximal and, actually, we have the stronger result that $\ol{\rn}^c\subseteq \rho(T)$ (\hyperref[th_strip_op_assoc]{Theorem} \ref{th_strip_op_assoc}).

\section{$\rn$-maximal operators}
\label{sec:max_oper}

In this paper we indicate by $\H$ a complex Hilbert space, with inner product $\pint$ and norm $\nor$. The domain and the range of a (linear) operator $T$ are denoted by $D(T)$ and $R(T)$, respectively. The set of bounded operators defined on the whole of $\H$ is denoted by $\B(\H)$.\\

\no	For the notion and properties of the numerical range we refer the reader to \cite[Ch. V, Sect. 3]{Kato}, \cite[Ch. 2]{Schm} and also \cite{Gust_Rao,Halmos} for bounded operators. \\
The {\it numerical range} of an operator $T$ on $\H$, with domain $D(T)$, is the convex subset of the complex plane
$$
\rn_T:=\{\pin{T\xi}{\xi}:\xi \in D(T), \n{\xi}=1\}.
$$
If $T$ is densely defined and $\rn_T\sub \R$ (or equivalently $T\sub T^*$) then it is called {\it symmetric}, while if $\rn_T\subseteq \{\lambda\in \C:\Re \lambda\geq 0\}$ then $T$ is said {\it accretive}.\\
If $T$ is bounded then $\rn_T$ is bounded, and the converse is true provided that $T$ is densely defined.\\

\no	Suppose that $\rn_T\neq \C$. Since $\rn_T$ is a convex subset of $\C$ then the complement $\ol{\rn_T}^c$ is connected or it consists of two half-planes (this second possibility holds if and only if $\ol{\rn_T}$ is a strip, i.e., a subset bounded by two parallel straight lines).\\
Assume that $T$ is closable and $\lambda \in \ol{\rn_T}^c$, then the number $\dim R(T-\lambda I)^\perp$ is constant in each connected component of $\ol{\rn_T}^c$, and it is called a {\it defect index} of $T$. Therefore, for an operator whose numerical range is not $\C$ the defect indexes are at least one and at most two (note that, actually, the defect index is defined and is constant in each connected component of the so-called regularity domain \cite[Definition 2.1]{Schm}; however, we are interested only in defect indexes defined outside the numerical range).\\
Finally, we recall some results involving the numerical range, the resolvent and the spectrum of an operator (see \cite[Ch. 22]{Halmos} and \cite[Ch. 2]{Schm}).

\begin{lem}
	\label{lem_intro}
	Let $T$ be an operator on $\H$, with numerical range $\rn_T$, resolvent set $\rho(T)$, point spectrum $\sigma_p(T)$, continuous spectrum $\sigma_c(T)$ and residual spectrum $\sigma_r(T)$. Then, the following assertions hold:
	\begin{enumerate}
		\item $\sigma_p(T)\subseteq \rn_T$;
		\item $\sigma_c(T)\subseteq \ol{\rn_T}$;
		\item each connected component of $\ol{\rn_T}^c$ is entirely contained in $\rho(T)$ or in $\sigma_r(T)$;
		\item if $T\in \B(\H)$ then $\sigma_r(T) \subseteq \rn_T$, i.e., the spectrum of $T$ is contained in $\ol{\rn_T}$;
		\item if $\lambda\in \ol{\rn_T}^c \cap \rho(T)$, then $\n{(T-\lambda I)^{-1}}\leq (\text{dist}(\lambda,\ol{\rn_T}))^{-1}$.
	\end{enumerate}
\end{lem}

\no	Now we give a new definition. Throughout this paper, if not otherwise specified, we assume that  $\rn \subset \C$ is a nonempty, convex, proper subset of $\C$.

\begin{defin}
	\label{def_n_max}
	An operator $T$, with numerical range $\rn_T$, is said to be {\it $\rn$-maximal} if the following conditions hold:
	\begin{enumerate}
		\item $\rn_T\subseteq \rn$;
		\item if $T'$ is an operator, with numerical range $\rn_{T'}$, such that $T\subseteq T'$ and $\rn_{T'}\subseteq \rn$, then $T=T'$.
	\end{enumerate}
\end{defin}

\begin{osr}
	Maximal accretive, and maximal symmetric operators are special cases of $\rn$-maximal operators, that is they are obtained considering $\rn=\{\lambda\in \C:\Re \lambda \geq 0\}$ and $\rn=\R$, respectively.
\end{osr}

\no	In \cite{Neumann} von Neumann proved the next celebrated characterization (see also \cite[Ch. 13]{Schm}).

\begin{teor}
	\label{car_simm}
	Let $T$ be a symmetric operator on $\H$. The following statements are equivalent:
	\begin{enumerate}
		\item $T$ is maximal symmetric;
		\item $R(T-i I )=\H$ or $R(T+i I )=\H$ (i.e., a defect index of $T$ is $0$);
		\item a connected component of $\C\backslash\R$ is contained in the resolvent set $\rho(T)$ of $T$.
	\end{enumerate}
	Moreover $T$ is self-adjoint if, and only if, $R(T-i I )=R(T+i I )=\H$ if, and only if, $\C\backslash\R\subseteq \rho(T)$.
\end{teor}

\no		We notice here that a similar characterization of selfadjointness has been given in \cite[Theorem 5.1]{Seb_Tarc} which simultaneously concerns  both real and complex Hilbert spaces. The next characterization (which follows the framework of symmetric operator) covers accretive operators and is due to Phillips \cite{Phillips}. Actually, Phillips worked with dissipative operators $T$, i.e., operators with numerical range in $\{\lambda \in \C: \Re \lambda \leq 0\}$. But, since $-T$ is accretive, the result for accretive operators follows easily from the dissipative case.

\begin{teor}[{\cite[Ch. I]{Phillips}}]
	\label{car_accr}
	Let $T$ be an accretive operator on $\H$. The following statements are equivalent:
	\begin{enumerate}
		\item $R(T-\lambda I )=\H$ for some $\lambda\in \C$ with $\Re \lambda <0$;
		\item the half-plane $\{\lambda \in \C: \Re \lambda <0\}$ is contained in the resolvent set $\rho(T)$ of $T$ (i.e., the defect index of $T$ is $0$);
		\item $T$ is maximal accretive and densely defined;
		\item $T$ is maximal accretive and closed.
	\end{enumerate}
\end{teor} 

\begin{osr}
	\label{oss_max_n_n2}
	Let $\rn'\sub \rn\subset \C$ be two proper, convex subsets of $\C$. An operator $T$ on $\H$, with numerical range in $\rn'$ and $\rn$-maximal is also $\rn'$-maximal.	The converse is not true in general. Indeed, consider $\rn':=\{\lambda\in \C:\Re \lambda =0\}$ contained in  $\rn:=\{\lambda \in \C: \Re \lambda \geq 0\}$. 	Let $T$ be a maximal symmetric operator, but not self-adjoint. We can assume that, in particular, $R(T-iI)\neq \H$.	Thus $\mathcal{T}:=iT$ is densely defined, has numerical range in $\rn'$ (hence it is in particular accretive), it is $\rn '$-maximal and one has $R(\mathcal{T}+I)\neq\H$. By \hyperref[car_accr]{Theorem} \ref{car_accr}, $\mathcal{T}$ is not $\rn$-maximal.
\end{osr}

\section{Case 1: general closed convex subset}

Our goal in the ensuing sections is to extend \hyperref[car_accr]{Theorem} \ref{car_accr} to $\rn$-maximal operators, where $\rn$ is a proper convex subset of $\C$. 

\begin{pro}
	\label{pro_rn_max_dens}
	Let $T$ be an  operator on $\H$ with numerical range contained in $\rn$ and $\lambda \in \rn^c$. If $R(T-\lambda I)=\H$, then $T$ is densely defined, closed and $\rn$-maximal.
\end{pro}
\begin{proof}
	Since $\lambda \notin \rn$, $T-\lambda I$ is injective and therefore $\lambda\in \rho(T)$ and $T$ is closed. 
	Let $T'$ be an extension of $T$, with numerical range $\rn_{T'}\subseteq \rn$. Then also $T'-\lambda I$ is injective; hence $T'=T$ and $T$ is $\rn$-maximal.\\
	Now, we prove that $T$ is densely defined. Let $\eta\in \H$ be such that $\pin{\xi}{\eta}=0$ for all $\xi \in D(T)$. Since $\lambda \in \rho(T)$, we have in particular that
	$
	\pin{(T-\lambda I)^{-1}\eta}{\eta}=0
	$ 
	and setting $\chi=(T-\lambda I)^{-1}\eta$, we get
	$
	\pin{\chi}{(T-\lambda I)\chi}=0
	$.
	Hence $\chi=0$, because $\lambda \notin \rn$, and $\eta=0$.
\end{proof}

\begin{esm}
	\label{esm_op_molt_n_max}
	Let $\ull{\alpha}=\{\alpha_n\}$ be a sequence of complex numbers contained in a closed convex set in $\rn$, and $\ll_2$ be the Hilbert space of the complex sequences $\{\xi_n\}$ such that $\sum_{n=1}^{\infty} |\xi_n|^2 <\infty$, with the usual inner product.\\
	Consider the operator $\Ma$ on $l_2$ with domain 
	$$
	D(\Ma)=\left \{ \{\xi_n\}\in \ll_2: \sum_{n=1}^{\infty } |\alpha_n \xi_n|^2<\infty \right\}
	$$
	and given by $\Ma \{\xi_n\}=\{\alpha_n \xi_n\}$, for all $\{\xi_n\}\in D(\Ma).$\\
	The operator $\Ma$ has numerical range in $\rn$; moreover if $\lambda \in \rn^c$, then $\lambda \in \rho(\Ma)$,  hence $R(\Ma-\lambda I)={\it l}_2$. By \hyperref[pro_rn_max_dens]{Proposition} \ref{pro_rn_max_dens}, $\Ma$ is $\rn$-maximal.
\end{esm}

\begin{pro}
	\label{pro_rn_max_chius}
	If $\rn$ is closed, then a densely defined, $\rn$-maximal operator on $\H$ is closed.  
\end{pro}
\begin{proof}
	Let $T$ be a densely defined, $\rn$-maximal operator on $\H$. $T$ is closable by \cite[Ch. V, Th. 3.4]{Kato}, and it has a closure $\ol{T}$ with numerical range $\rn_{\ol{T}}\sub  \ol{\rn_T}\sub \ol{\rn}=\rn$. Therefore, by the maximality of $T$ and from $T\sub \ol{T}$, we have $T=\ol{T}$.
\end{proof}

\no	Resuming the results obtained, and using %Lemma $\ref{cor_comp_con_ris}$
\hyperref[lem_intro]{Lemma} \ref{lem_intro}, we can formulate the following theorem.

\begin{teor}
	\label{th_car_parz_n_max}
	Let $\rn\sub \C$ be a proper, closed, convex subset of $\C$ and let $T$ be an operator on $\H$ with numerical range in $\rn$. \\
	For the following statements
	\begin{enumerate}
		\item $R(T-\lambda I)=\H$ for some $\lambda\in \rn^c$;
		\item a connected component of $\rn^c$ is contained in the resolvent set $\rho(T)$ of $T$ (i.e. a defect index of $T$ is $0$);
		\item $T$ is $\rn$-maximal and densely defined;
		\item $T$ is $\rn$-maximal and closed;
	\end{enumerate}
	the following implications hold ${ 1.\Rightarrow 2.\Rightarrow 3.\Rightarrow 4.}$
\end{teor}

\no	The statements indicated in the previous theorem are equivalent in the case  $\rn=\{\lambda\in \C:\Re\lambda \geq 0\}$ (\hyperref[car_accr]{Theorem} \ref{car_accr}), but they are not equivalent in general (see \hyperref[count_exm_n_max_clos]{Example} \ref{count_exm_n_max_clos} below).
If $T$ is bounded, then the statements in \hyperref[th_car_parz_n_max]{Theorem} \ref{th_car_parz_n_max} are equivalent and moreover they hold if and only if $T\in \B(\H)$. 

\section{Case 2: closed convex subset of a sector}
\label{sec:sect}

In this section we study the case in which $\rn$ is a closed subset of a sector 
$$\Set:=\{\lambda\in \C:|\arg(\lambda-\gamma)|\leq \theta\}$$
where $\gamma\in \R,\theta\in [0,\frac{\pi}{2})$, and we add other implications to \hyperref[th_car_parz_n_max]{Theorem} \ref{th_car_parz_n_max}. \\
We recall that an operator $T$ with numerical range $\rn_T\subseteq \Set$ is said to be {\it sectorial}; moreover, if the complement of $\Set$ is contained in the resolvent set of $T$, then $T$ is said {\it m-sectorial} (see \cite[Ch. V]{Kato}).

\begin{teor}
	\label{th_car_n_max_set}
	Let $\rn \subset \C$ be a closed, convex subset contained in a sector of $\C$, and let $T$ be an operator on $\H$ with numerical range $\rn_T\subseteq \rn$. The following statements are equivalent:
	\begin{enumerate}
		\item $T$ is $\rn$-maximal and densely defined;
		\item $R(T-\lambda I)=\H$ for some $\lambda\in {\rn}^c$;
		\item $\rn^c$ is contained in the resolvent set $\rho(T)$ of $T$ (i.e., the defect index of $T$ is $0$).
	\end{enumerate}
	If these conditions are satisfied, then $T$ is closed.
\end{teor}
\begin{proof}
	By \hyperref[th_car_parz_n_max]{Theorem} \ref{th_car_parz_n_max}, we only have to prove that if $T$ is $\rn$-maximal and densely defined, then $R(T-\lambda I)=\H$ for some $\lambda\in {\rn}^c$.\\
	Let $T'$ be the (m-sectorial) Friedrichs extension of $T$. The numerical range $\rn_{T'}$ of $T'$ is contained in $\overline{\rn_T}\subseteq \rn$ (see \cite[Ch. VI]{Kato}). 
	Hence, from the $\rn$-maximality of $T$ and $T\subseteq T'$, we have $T=T'$. It follows that $T$ is m-sectorial, i.e., $R(T-\lambda I)=\H$ for some $\lambda\in {\rn}^c$.
\end{proof}

\no	The following theorem demonstrates that the $\rn$-maximality of an operator (where $\rn$ is contained in a sector) does not strictly depend  on the chosen closed, convex subset $\rn$.

\begin{teor}
	\label{unic_est}
	Let $\rn_1,\rn_2 \subset \C$ be two proper, closed, convex subsets of $\C$, such that $\rn_1$ is contained in a sector of $\C$ and $\rn_1\cap \rn_2 \neq \varnothing$. Let $T$ be a densely defined operator on $\H$ with numerical range $\rn_T\subseteq \rn_1\cap \rn_2$. The following statements are equivalent:
	\begin{enumerate}
		\item $T$ is $\rn_1$-maximal;
		\item $T$ is $\rn_2$-maximal.
	\end{enumerate}
\end{teor}
\begin{proof}
	$(1.\Rightarrow 2.)$ Since $\rn_1$ is contained in a sector, we have $\rn_1^c\cap \rn_2^c\neq \varnothing$.	By \hyperref[th_car_n_max_set]{Theorem} \ref{th_car_n_max_set}, $R(T-\lambda I)=\H$ for all $\lambda\in \rn_1^c$, hence $R(T-\lambda I)=\H$ per some $\lambda \in \rn_2^c$. Applying \hyperref[th_car_parz_n_max]{Theorem} \ref{th_car_parz_n_max}, $T$ is $\rn_2$-maximal.\\
	$(2.\Rightarrow 1.)$ By \hyperref[oss_max_n_n2]{Remark} \ref{oss_max_n_n2}, $T$ is $(\rn_1\cap \rn_2)$-maximal, and then $T$ is $\rn_1$-maximal using the first implication ($\rn_1\cap \rn_2$ is contained in a sector).
\end{proof}

\no \hyperref[oss_max_n_n2]{Remark} \ref{oss_max_n_n2} shows that \hyperref[unic_est]{Theorem} \ref{unic_est} does not hold without the hypothesis that $\rn_1$ is contained in a sector of $\C$. Another way to read \hyperref[unic_est]{Theorem} \ref{unic_est} is the next corollaries.

\begin{cor}
	Let $\rn \subset \C$ be a closed, convex subset contained in a sector of $\C$. A $\rn$-maximal, densely defined operator $T$ on $\H$ has no proper extension whose numerical range is a proper subset of $\C$.
\end{cor}

\begin{cor}
	Let $T$ be a densely defined, accretive operator and with numerical range also contained in a closed subset $\rn$ of a sector $\Set$ of $\C$. Then, $T$ is maximal accretive if and only if $T$ is $\rn$-maximal.
\end{cor}

\no	The positive semi-line is contained in some sector of $\C$. For this reason, we turn our attention to the case in which $T$ is {\it positive}, i.e., $\rn_T\subseteq [0+\infty)$. We prefer to say that $T$ is {\it maximal positive} if it is $[0,+\infty)$-maximal. Before to show how \hyperref[th_car_parz_n_max]{Theorem} \ref{th_car_parz_n_max} is formulated in this case, we recall that 
a closed positive operator $T$ is said {\it positively closable} (see \cite{Ando-Nishio}) if 
$
\displaystyle \lim_{n\to \infty} \pin{T \xi_n}{\xi_n}=0$ and $\displaystyle \lim_{n\to \infty} T\xi_n=\eta
$
implies $\eta=0$. 

\begin{teor}
	\label{th_car_pos}
	Let $T$ be a positive operator on $\H$. The following statements are equivalent:
	\begin{enumerate}
		\item $T$ is maximal positive, closed and positively closable;
		\item $[0,+\infty)^c$ is contained in the resolvent set $\rho(T)$ of $T$ (i.e., the defect index of $T$ is $0$);
		\item $T$ is maximal positive and densely defined;
		\item $R(T-\lambda I)=\H$ for some $\lambda\in [0,+\infty)^c$.	
	\end{enumerate}
\end{teor}	
\begin{proof}
	Suppose that $T$ is maximal positive, closed and positively closable. By \cite[Theorem 1]{Ando-Nishio}, $T$ admits a positive self-adjoint extension, that must concides with $T$; hence $[0,+\infty)^c$ is contained in the resolvent set $\rho(T)$ of $T$.	The other implications follow by \hyperref[th_car_parz_n_max]{Theorem} \ref{th_car_parz_n_max}.
\end{proof}

\no	The next example shows that \hyperref[th_car_pos]{Theorem} \ref{th_car_pos} does not hold without the hypothesis that $T$ is positively closable. That is the statements in \hyperref[th_car_parz_n_max]{Theorem} \ref{th_car_parz_n_max} are not equivalent in general.

\begin{esm}
	\label{count_exm_n_max_clos}
	Let $\H=\C^2$ and $T$ be the operator on $\C^2$ with domain 
	$D(T)=\{(x,0):x \in \C\}$ 
	and defined by $T(x,0)=(0,x)$ for all $x\in \C$. 
	We have that $T$  is positive, closed and non densely defined, $R(T-\lambda I)\neq \C^2$ for all $\lambda\in [0,+\infty)^c $. Moreover, $T$  is not positively closable, then by \cite[Theorem 1]{Ando-Nishio} is maximal positive.
\end{esm}

\section{Case 3: closed strip}	
\label{sec:strip}	

Now we study the case where the set $\rn$ of \hyperref[th_car_parz_n_max]{Theorem} \ref{th_car_parz_n_max}  is a strip. More precisely, we consider the following subsets of $\C$:
\begin{enumerate}
	\item for $\alpha \geq 0$, the {\it horizontal strip} $\St_\alpha$, i.e., a subset   such that 
	$$\{\lambda \in \C:|\Im \lambda|< \alpha\}\sub \St_\alpha\sub \{\lambda \in \C:|\Im \lambda|\leq \alpha\};$$ 
	\item the {\it horizontal closed strip} $\ol{\St_\alpha}$, i.e., a subset $\ol{\St_\alpha}:=\{\lambda \in \C:|\Im \lambda|\leq \alpha\}$, where $\alpha \geq 0$;
	\item the {\it strip} $\St$, i.e., a subset $\St:=a\St_\alpha+b$, where $a\in \C\backslash\{0\}, b\in \R,\alpha \geq 0$.
	\item the {\it closed strip} $\ol{\St}$, i.e., a subset $\ol{\St}:=a\ol{\St_\alpha}+b$, where $a\in \C\backslash\{0\}, b\in \R,\alpha \geq 0$.
\end{enumerate}

\no	We recall some notions regarding sesquilinear forms (see \cite[Ch. VI]{Kato}), that are useful in this section, but also in the last one.\\	
Let $\D$ be a subspace of the Hilbert space $\H$ and let $\O$ be a sesquilinear form defined on $\D$. The {\it adjoint} $\O^*$ of $\O$ is the form on $\D$ given by
$$
\O^*(\xi,\eta)=\ol{\O(\eta,\xi)}, \qquad \forall \xi,\eta \in \D.
$$
If $\O=\O^*$ then $\O$ is said to be {\it symmetric} . The symmetric sesquilinear forms on $\D$ defined by
$$
\Re \O=\frac{1}{2}(\O+\O^*) \qquad \text{and} \qquad \Im \O=\frac{1}{2i}(\O-\O^*),
$$
are called {\it real} and {\it imaginary parts} of $\O$, respectively; then $\O=\Re\O+i\Im\O$. \\
The {\it numerical range} is defined also for a sesquilinear form $\O$ and it is the convex subset  $$
\rn_\O:=\{\O(\xi,\xi):\xi\in \D, \n{\xi}=1\}
$$ of $\C$. Note that $\O$ is bounded if and only if $\rn_\O$ is bounded; $\O$ is symmetric if and only if $\rn_\O\subseteq \R$. If $\O$ is bounded and $\D=\H$, then there exists a unique operator $B\in \B(\H)$ such that 
$\O(\xi,\eta)=\pin{B\xi}{\eta}$,  for all $\xi, \eta \in \H$.\\
In order to prove \hyperref[th_car_strip]{Theorem} \ref{th_car_strip} we firstly need the next lemma. The idea of the proof is analogous to the argument used to prove Theorem 7.1.2 of \cite{Haase_b}.

\begin{lem}
	\label{lem_decomp_strip}
	Let $\St_\alpha$ be a horizontal strip of $\C$ and $T$ be a densely defined operator with numerical range $\rn_T \subseteq \St_\alpha$. Then there exist unique symmetric operators $B\in \B(\H)$ and $S$ such that $D(S)=D(T)$ and
	\begin{equation}
	\label{T_strip}
	T=S+iB.
	\end{equation}
	Moreover,
	\begin{enumerate}
		\item $D(T)\subseteq D(T^*)$;
		\item $S=\frac{1}{2} (T+T^*)$ and $B_{|D(T)}=\frac{1}{2i} (T-T^*)$.
	\end{enumerate}
\end{lem}
\begin{proof}
	Let $\O$ be the sesquilinear form on $D(T)$ given by
	$$
	\O(\xi,\eta)=\pin{T\xi}{\eta}, \qquad \forall \xi, \eta \in D(T).
	$$
	Consider the real and imaginary parts $\Re \O$, $\Im \O$ of $\O$. The numerical range of $\O$ is exactly the one of $T$, so, from $\O=\Re \O+i\Im \O$, we have that $\Re\O$ and $\Im \O$ have numerical ranges in $\R$ and in $[-\alpha,\alpha]$, respectively. \\
	Consequently, $\Im \O$ is bounded, and since it is densely defined, it can be extended to a unique bounded form in whole $\H$. Hence, there exists a unique (symmetric) operator $B\in \B(\H)$ such that $\Im\O(\xi,\eta)=\pin{B\xi}{\eta}$, for all $\xi,\eta \in D(T)$. Now set $S:=T-iB$, hence $D(S)=D(T)$. We have, for $\xi \in D(S)$ with $\n{\xi}=1$,
	$$
	\pin{S\xi}{\xi}=\pin{T\xi}{\xi}-i\pin{B\xi}{\xi}=\O(\xi,\xi)-i\Im(\xi,\xi)=\Re(\xi,\xi) \in \R,
	$$
	therefore $S$ is symmetric. \\
	To prove {\it 1.} observe that $T^*=S^*-iB$, so $D(T)=D(S)\subseteq D(S^*)=D(T^*)$. This implies that $T+T^*$ is defined on $D(T)$ and 
	$$
	T+T^*=S+S^*=2S, \qquad \text{on $D(T)=D(S)$},
	$$
	hence $S=\frac{1}{2} (T+T^*)$. In a similar way it can be verified that $B_{|D(T)}=\frac{1}{2i} (T-T^*)$, which proves statement 2.\\
	Suppose now that $T=S'+iB'$, with $S',B'$ symmetric operators, $D(S')=D(T)$ and $B'\in \B(\H)$. It follows that $S-S'=-i(B-B')$, but both $S-S'$ and $B-B'$ are symmetric, therefore $S=S'$ and $B=B'$.
\end{proof}

\no	Denote by $S(\H)$ the family of symmetric operators on $\H$ and $St(\H)$ the family of densely defined operators on $\H$ with numerical range in a strip $\St_\alpha$. Thus, with the aid of the previous lemma, we can formulate the following correspondence and its properties.

\begin{cor}
	The map $S(\H)\times \B(\H) \to St(\H)$ defined by	$(S,B)\mapsto S+iB$ 
	is a bijection.
\end{cor}

\begin{lem}
	\label{corr_extens_strip}
	Let $\St_\alpha$ be a horizontal strip, $T$ a densely defined operator with numerical range $\rn_T \subseteq \St_\alpha$ and $T=S+iB$ the decomposition \emph{(\ref{T_strip})}. \\
	The map $S'\mapsto T':=S'+iB$ defines 
	\begin{enumerate}
		\item a one-to-one correspondence between all extensions $S'$ of $S$ and all extensions $T'$ of $T$;
		\item a one-to-one correspondence between all symmetric extensions $S'$ of $S$ and all extensions $T'$ of $T$ with numerical range $\rn_{T'}\subseteq \ol{\St_\alpha}$.
	\end{enumerate}
\end{lem}
\begin{proof}
	The first statement is obvious. Let $S'$ be a symmetric extension of $S$, then, clearly, $T':=S'+iB$ is an extension of $T$ whose numerical range satisfies $\rn_{T'}\subseteq \ol{\St_\alpha}$.\\  
	Now, let $T'$ be an extension of $T$ with numerical range $\rn_{T'}\subseteq \ol{\St_\alpha}$, and $T'=S'+iB'$ the decomposition given by \hyperref[lem_decomp_strip]{Lemma} \ref{lem_decomp_strip}. Since $T\subseteq T'$ then, following the proof of the same lemma, $B=B'$; hence $S=T-iB\subseteq T'-iB=S'$.
\end{proof}

\begin{cor}
	\label{cor_max_strip}
	Let $\St_\alpha$ be a horizontal strip, $T$ a densely defined operator with numerical range $\rn_T \subseteq \St_\alpha$ and $T=S+iB$ the decomposition \emph{(\ref{T_strip})}. 
	\begin{enumerate}
		\item If $S$ is maximal symmetric, then $T$ is $\St_\alpha$-maximal.
		\item If $\St_\alpha=\ol{\St_\alpha}$ is closed, then $T$ is $\ol{\St_\alpha}$-maximal if and only if $S$ is maximal symmetric.
	\end{enumerate}
\end{cor}

\no \hyperref[th_car_parz_n_max]{Theorem} \ref{th_car_parz_n_max} is adapted to the case of a strip as follows.

\begin{teor}
	\label{th_car_strip}
	Let $\ol{\St}$ be a closed strip of $\C$ and $T$ an operator on $\H$ with numerical range in $\ol{\St}$. The following statements are equivalent:
	\begin{enumerate}
		\item $R(T-\lambda I)=\H$ for some $\lambda\in \ol{\St}^c$;
		\item a connected component of $\ol{\St}^c$ %(i.e., $\{\lambda \in \C: \Im \lambda >\alpha\}$ or $\{\lambda \in \C: \Im \lambda <-\alpha\}$) 
		is contained in the resolvent set $\rho(T)$ of $T$ (i.e., a defect index of $T$ is $0$);
		\item $T$ is $\ol{\St}$-maximal and densely defined.
	\end{enumerate}
	If these conditions are satisfied, then $T$ is closed.
\end{teor}
\begin{proof}
	We only have to prove the implication {\it 3. $\Rightarrow$ 1.} by \hyperref[th_car_parz_n_max]{Theorem} \ref{th_car_parz_n_max}. \\
	With a linear transformation (which does not change the maximality), we can restrict ourselves to the case in which $\ol{\St}$ is horizontal, i.e., $\ol{\St}=\{\lambda \in \C:|\Im \lambda|\leq \alpha\} $, for some $\alpha\geq0$.\\
	%The case $\alpha=0$ is trivial. % since $\ol{\St_\alpha}=\R$. \\
	Let $T=S+iB$ be the decomposition (\ref{T_strip}). The case $B=0$ is trivial. Assume $B\neq 0$, hence $\alpha>0$.	By \hyperref[cor_max_strip]{Corollary} \ref{cor_max_strip} $S$ is maximal symmetric, hence we can find $\lambda \in \rho(S)$ such that $|\Im\lambda|>\alpha $. As proved in the proof of \hyperref[lem_decomp_strip]{Lemma} \ref{lem_decomp_strip}, $B$ has numerical range in $[-\alpha,\alpha]$; this implies that $\n{B}\leq \alpha$. We also have
	$
	\n{(S-\lambda I)^{-1}}\leq |\Im \lambda|^{-1},
	$
	therefore
	$
	\n{(S-\lambda I)^{-1}}\leq |\Im \lambda|^{-1}<\alpha^{-1} \leq \n{B}^{-1}.
	$
	By \cite[Theorem 5.11]{Weid}, $\lambda \in \rho(T)$.
\end{proof}

\begin{cor}
	\label{cor_strip_auto}
	Let $T$ be a densely defined operator with numerical range contained in a closed strip $\ol{\St}$. Then $D(T)=D(T^*)$ if and only if $\ol{\St} ^c \subseteq \rho(T)$.
\end{cor}

\begin{proof}
	It is not restrictive that we consider $\ol{\St}=\{\lambda \in \C:|\Im \lambda|\leq \alpha\} $, for some $\alpha\geq 0$. Let $T=S+iB$ be the decomposition (\ref{T_strip}). By \hyperref[lem_decomp_strip]{Lemma} \ref{lem_decomp_strip} $D(T)=D(T^*)$ if and only if $D(S)=D(S^*)$ if and only if $S$ is self-adjoint. But, with an argument like the one used in the proof of \hyperref[th_car_strip]{Theorem} \ref{th_car_strip}, $S$ is self-adjoint, if and only if $\ol{\St} ^c \subseteq \rho(T)$.
\end{proof}

\begin{osr}
	The sufficient implication of \hyperref[cor_strip_auto]{Corollary} \ref{cor_strip_auto}, in the case of horizontal strip, is also proved in \cite[Theorem 7.1.2]{Haase_b}.
\end{osr}

\begin{pro}
	Let $\rn \subset \C$ be a proper, convex subset of $\C$, $\ol{\St}$ be a closed strip, such that $\rn\cap \ol{\St} \neq \varnothing$ and $\rn$ does not contain any of two half-planes which constitute $\ol{\St}^c$. Moreover, let $T$ be a densely defined operator on $\H$ with numerical range $\rn_T\subseteq \rn\cap \ol{\St}$. If $T$ is $\ol{\St}$-maximal, then $T$ is $\rn$-maximal.
\end{pro}
\begin{proof}
	By \hyperref[th_car_strip]{Theorem} \ref{th_car_strip}, $R(T-\lambda I)=\H$ for all $\lambda$ contained in a connected component of $\ol{\St}^c$ (i.e., one of the two half-planes which constitute $\ol{\St}^c$). By the hypothesis and applying \hyperref[pro_rn_max_dens]{Proposition} \ref{pro_rn_max_dens}, $T$ is $\rn$-maximal.
\end{proof}

\begin{esm}
	Let $AC[a,b]$ be the set of absolutely continuous function on an interval $[a,b]$,  $\mathcal{J}$ be one of the open intervals $(0,1),(0,\infty),\R$, and
	$$
	H^1(\mathcal{J})=\{f\in L^2(\mathcal{J}):f\in AC[a,b] \text{ for all } [a,b]\subseteq\ol{\mathcal{J}} \text{ and } f'\in L^2(\mathcal{J})\}
	$$
	$$
	H^1_0(0,1)=\{f\in H^1(0,1):f(0)=f(1)=0\}
	$$
	$$
	H^1_0(0,+\infty)=\{f\in H^1(0,+\infty):f(0)=0\}.
	$$	
	Consider the densely defined differential operator $T$ on $L^2(\mathcal{J})$ given by 
	$$
	(Tf)(x)=i \left (f'(x) + r(x)f(x) \right), \qquad \forall f\in H^1_0(\mathcal{J}),
	$$
	on the domain $D(T)=H^1_0(\mathcal{J})$ if $\mathcal{J}=(0,1)$ or $\mathcal{J}=(0,+\infty)$, or on the domain $D(T)=H^1(\mathcal{J})$ if $\mathcal{J}=\R$, 
	where $r:\mathcal{J}\to \R$ is a bounded continuous function, i.e., there exists $m>0$ such that $|r(x)|\leq m$, for all $x \in \mathcal{J}$.\\
	The numerical range of $T$ is contained in the strip $\ol{\St_m}$.  Our goal is to find all the $\ol{\St}$-maximal extensions $T'$ of $T$, where $\ol{\St}$ is a closed horizontal strip containing $\ol{\St_m}$. Therefore, let $k\geq m$ and $T'$ be a $\ol{\St_k}$-maximal extension of $T$. \\
	Clearly, the decomposition of \hyperref[lem_decomp_strip]{Lemma} \ref{lem_decomp_strip} is $T=S+iB$, where $S$ is the symmetric operator with domain $D(T)$ defined by 
	$$
	(Sf)(x)=if'(x), \qquad \forall f\in D(T),
	$$
	and $B$ is the bounded symmetric operator on $L^2(\mathcal{J})$ given by
	$$
	(Bf)(x)=r(x) f(x), \qquad \forall f\in L^2(\mathcal{J}).
	$$
	On the other hand, $T'=S'+B'$ by \hyperref[lem_decomp_strip]{Lemma} \ref{lem_decomp_strip} where, in particular, $S'$ is maximal symmetric by \hyperref[cor_max_strip]{Corollary} \ref{cor_max_strip}. Since $T\sub T'$, we have $B=B'$ and $S\sub S'$ by \hyperref[corr_extens_strip]{Lemma} \ref{corr_extens_strip}.\\
	It is well-known (see \cite[Sect. 13.2]{Schm}) that $S$ is closed and has defect indexes $d_+=\dim R(S+i I)^\perp$ and $d_-=\dim R(S-i I)^\perp$:
	\begin{enumerate}
		\item $d_+=d_-=1$, if $\mathcal{J}=(0,1)$;
		\item $d_+=1$, $d_-=0$ (and hence $S$ is maximal symmetric), if $\mathcal{J}=(0,+\infty)$;
		\item $d_+=d_-=0$ (i.e., $S$ is self-adjoint), if $\mathcal{J}=(-\infty,+\infty)$.
	\end{enumerate}
	It follows that $T=S+iB$ is $\ol{\St_k}$-maximal in the cases $\mathcal{J}=(0,+\infty)$ and $\mathcal{J}=(-\infty,+\infty)$. 
	Conversely, in the case $\mathcal{J}=(0,1)$, all the maximal symmetric  extensions (that are also self-adjoint) of $S$ are the operators $S_\theta$ (where $\theta$ is a complex number of modulus $1$) with domains 
	$
	D(S_\theta)=\{f\in H^1(0,1):f(-1)=\theta f(1)\},
	$
	and given by $	(S_\theta f)(x)=if'(x)$, for all $f\in D(S_\theta)$. Consequently, for some $\theta \in \C$ with $|\theta|=1$, $T'$ is the operator defined on the domain $D(T')=D(S_\theta)$ as
	$$
	(T' f)(x)=i \left (f'(x) + r(x) f(x) \right), \qquad \forall f\in D(T_\theta).
	$$
	Note that in all cases $T'$ has numerical range in the smaller strip $\ol{\St_m}$, hence all $\ol{\St_k}$-maximal extension of $T$ are actually $\ol{\St_m}$-maximal.
\end{esm}

\section{$\rn$-maximal operators as generators of semi-groups}
\label{sec:exampl}	

In this section we report some assertions (in part well-known) regarding generators of semi-groups on $\H$ which are $\rn$-maximal, with some proper, convex subsets $\rn$. \\

\no Let $\{S(t)\}_{t\geq 0}$ be a strongly continuous semi-group of bounded operators on $\H$ and let $A$ be its generator.  We recall that (\cite[Ch. I, Th. 2.2]{Pazy}) there exist constants $M\geq 1, \omega \geq 0$ such that 
\begin{equation}
\label{bound_semigroup}
\n{S(t)}\leq Me^{\omega t}, \qquad  \forall t\geq 0.
\end{equation}
Moreover, if the semi-group extends to a strongly continuous group $\{S(t)\}_{t\in \R}$, then 
there exist constants $M\geq 1, \omega \geq 0$ such that 
\begin{equation}
\label{bound_group}
\n{S(t)}\leq Me^{\omega |t|}, \qquad  \forall t\in \R.
\end{equation}

\begin{enumerate}
	\item The Lumer-Phillips theorem (\cite[Ch. I, Th. 4.3]{Pazy}) states that $\{S(t)\}_{t\geq 0}$ is a semi-group of contractions if and only if $A$ is a densely defined maximal dissipative.
	\item An immediate consequence of point 1 is that a semi-group $\{S(t)\}_{t\geq 0}$ satisfies $\n{S(t)}\leq e^{\omega t}$ for some $\omega \in \R$ and for all $t\geq 0$ if and only if $A$ is $\rn$-maximal and densely defined, where $\rn:=\{\lambda \in \C:\Re \lambda \leq \omega \}$ (see \cite[Ch. II, Ex. 2.2]{Engel-Nagel}).
	\item As proved in \cite[Theorem 1.1.4]{Phillips}, $\{S(t)\}_{t\geq 0}$ is a semi-group of isometries if and only if the numerical range of $A$ is contained in $\rn:=\{\lambda \in \C:\Re\lambda =0\}$ and $A$ is maximal dissipative and densely defined. This implies, in particular, that $A$ is $\rn$-maximal.
	\item Another consequence of Lumer-Phillips theorem establishes that $\{S(t)\}_{t\geq 0}$ extends to a strongly continuous group $\{S(t)\}_{t\in \R}$ and $\n{S(t)}\leq e^{\omega |t|}$ for all $t\in \R$ if and only if $A$ is $\rn$-maximal, densely defined and such that $\rn^c\subseteq \rho(A)$, where $\rn:=\{\lambda\in \C:|\Re \lambda| \leq \omega\}$.
	\item A more general case of point 3 and 4 is that
	\begin{equation}
	\label{pass_semigr_strip}
	e^{\omega_1 t}\leq \n{S(t)\xi}\leq e^{\omega_2 t}, \qquad  \forall t\geq 0,\forall \xi\in \H,\n{\xi}=1,
	\end{equation}
	for some $\omega_1\leq \omega_2$ if and only if $A$ is $\rn$-maximal and densely defined where $\rn:=\{\lambda\in \C:\omega_1 \leq \Re \lambda \leq \omega_2\}$. In fact we have for $\xi\in \H$ and $t\geq 0$, 
	$$
	2\Re \pin{AS(t)\xi}{S(t)\xi}=\frac{\partial}{\partial t} (\n{S(t) \xi}^2),
	$$
	i.e.,
	$$
	\Re \pin{AS(t)\xi}{S(t)\xi}=\n{S(t) \xi}\frac{\partial}{\partial t}\n{S(t) \xi},
	$$
	Hence, $A$ has numerical range in $\rn$ if and only if
	$$
	\omega_1 \n{S(t)\xi}\leq \frac{\partial}{\partial t} \n{S(t)\xi} \leq \omega_2 \n{S(t)\xi},
	$$
	i.e.,
	$
	Ce^{\omega_1 t}\leq \n{S(t)\xi} \leq Ce^{\omega_2 t},
	$
	for some $C\geq 0$. By $S(0)\xi=\xi$, we have $\n{\xi}e^{\omega_1 t}\leq \n{S(t)\xi}\leq \n{\xi}e^{\omega_2 t}$, for all $\xi \in \H, t\geq 0$. \\
	Since $A$ is a generator of a semi-group, then it is $\rn$-maximal and densely defined by \hyperref[th_car_strip]{Theorem} \ref{th_car_strip}. Moreover, $\rn^c\subseteq \rho(A)$ if and only if $\{S(t)\}_{t\geq 0}$ extends to a strongly continuous group, if and only if $S(t)$ has range $\H$ for all $t\geq 0$ (all $S(t)$ are injective by (\ref{pass_semigr_strip})).
	\item Assume that $\rn\subseteq (-\Set)$, where $\Set$ is a sector of $\C$. Thus, in particular, $-A$ is m-sectorial by \hyperref[th_car_n_max_set]{Theorem} \ref{th_car_n_max_set}, and hence $A$ generates a bounded holomorphic semi-group on $\H$ (see 
	\cite[Corollary 7.3.5]{Haase_b}).
\end{enumerate}

\no 	We can also state the following proposition that holds for a semi-group (resp. group) $\{S(t)\}_{t\geq 0}$ that does not satisfy condition (\ref{bound_semigroup}) (resp. (\ref{bound_group})) necessarily with $M=1$.	

\begin{pro}
	Let $A$ be an operator with numerical range $\rn_A\neq \C$.
	\begin{enumerate}
		\item If $A$ is the generator of a strongly continuous semi-group of bounded operators and $\rn_A$ does not contain any half-plane $\{\lambda\in \C:\Re\lambda > \omega\}$ with $\omega \geq 0$, then $A$ is $\rn_A$-maximal.
		\item If $A$ is the generator of a strongly continuous group of bounded operators, then $A$ is $\rn_A$-maximal.
	\end{enumerate}
\end{pro}
\begin{proof}
	\begin{enumerate}
		\item By \cite[Ch. II, Th. 3.8]{Engel-Nagel} the resolvent of $A$ contains the half-plane $H_\omega:=\{\lambda\in \C:\Re\lambda > \omega\}$, with a certain $\omega \geq 0$, and by hypothesis $H_\omega \cap \rn_A^c \neq \varnothing$. An application of \hyperref[pro_rn_max_dens]{Proposition} \ref{pro_rn_max_dens} shows that $A$ is $\rn_A$-maximal.
		\item 		This proof is analogous to the previous one. The difference is that the resolvent of $A$ contains the half-planes $-H_\omega$ and $H_\omega$, where $H_\omega:=\{\lambda\in \C:\Re\lambda > \omega\}$ and $\omega \geq 0$ (see \cite[Ch. II, Sect. 3]{Engel-Nagel}). The fact that $(-H_\omega \cup H_\omega) \cap \rn_A^c \neq \varnothing$ and \hyperref[pro_rn_max_dens]{Proposition} \ref{pro_rn_max_dens} imply that $A$ is $\rn_A$-maximal. \qedhere
	\end{enumerate}
\end{proof}

\no		\hyperref[lem_decomp_strip]{Lemma} \ref{lem_decomp_strip} establishes a decomposition of an operator in sum of real and imaginary parts. We mention \cite[Theorem 7.2.8]{Haase_b}, which states that if $A$ is the generator of a strongly continuous group $\{S(t)\}_{t\geq 0}$ with $\n{S(t)}\leq Me^{\omega_0 |t|}$, for all $t\in \R$, and $\omega>\omega_0$, then there exists a inner product $\pint_\circ$, inducing a norm $\nor_\circ$ equivalent to $\nor$, and with respect to $\pint_\circ$ the following statements hold:
\begin{enumerate}
	\item $A$ has numerical range in $\ol{\St_\omega}$ (i.e., $\pin{A\xi}{\xi}_\circ\in \ol{\St_\omega}$, for all $\xi \in \H$, $\n{\xi}_\circ=1$);
	\item denoting by $A^\circ$ the adjoint of $A$ with respect to $\pint_\circ$, we have $A=iB+C$ where
	\begin{itemize}
		\item $
		B=\frac{1}{2i}(A-A^\circ)$ and $C_{|D(A)}=\frac{1}{2}(A+A^\circ);
		$
		\item $B$ is self-adjoint and $D(B)=D(A)$;
		\item $C\in \B(\H)$ and it is symmetric.				
	\end{itemize}
\end{enumerate}
Since $A$ is the generator of a group, $A$ is $\ol{\St_\omega}$-maximal considering the inner product $\pint_\circ$. By \cite[Lemma C.5.2]{Haase_b}, we conclude the following (see also \cite[Theorem 2.4]{Delau}).

\begin{pro}
	The generator of a strongly continuous group of bounded operators is similar to a $\ol{\St_\omega}$-maximal operator, where $\ol{\St_\omega}$ is a horizontal closed strip.	
\end{pro}

\no	In several works, like \cite{Batty,Batty2,Delau,Haase_b,McIntosh_fc}, the authors defined a so-called {\it functional calculus} for a densely defined operator with spectrum contained in a subset $\rn$ which is a sector, a half-plane or a strip, and with resolvent operators satisfying some condition of boundedness. 	As particular case, it is possible to define a functional calculus for an operator $T$ with numerical range in $\rn$ and that satisfies $\ol{\rn}^c \subseteq \rho(T)$ (see \cite[Example 2.2.4, Section 2.3]{Batty} and \cite[Theorem 2.4]{Delau}).

\section{Correspondences with bounded operators}
\label{sec:corr}

It is worth mentioning that Phillips \cite{Phillips} proved \hyperref[car_accr]{Theorem}  \ref{car_accr} with the aid of the transform of an accretive operator $T$
\begin{equation}
\label{transf}
\tau(T)=(T-I)(T+I)^{-1},
\end{equation}
where $\tau(T)$ has domain $D(\tau(T))=R(T+I)$ and range $R(\tau(T))=R(T-I)$. Also von Neumann's \hyperref[car_simm]{Theorem} \ref{car_simm} can be proved with a similar map, more precisely with the {\it Cayley transform} of a symmetric operator $T$
$$
\kappa(T)=(T-iI)(T+iI)^{-1},
$$
with domain $D(\kappa(T))=R(T+iI)$ and range $R(\kappa(T))=R(T-iI)$ (see \cite[Ch. 13]{Schm}). 
Properties of transform (\ref{transf}) are settled in the next theorem.

\begin{teor}[{\cite[Sect. 1.1]{Phillips}}]
	\label{corr_trasf_accr}
	The transform $T\mapsto \tau(T)$ defines a one-to-one  correspondence, which preserves extensions, between all accretive operators on $\H$ and all contractions $J$ of $\H$ such that $I-J$ is invertible.\\
	In particular, the transform $T\mapsto \tau(T)$ defines a one-to-one correspondence between all densely defined, accretive operators on $\H$ and all contractions $J$ of $\H$ with $R(I-J)$ dense in $\H$.
\end{teor}

\no Let $T$ be an operator with domain $D(T)$ and numerical range contained in a proper, convex, subset $\rn$ of $\C$. We want to apply the method of the transform to $T$. Since $\rn$ is contained in a half-plane, then, up to linear transformation, we can assume that $\rn$ is contained in $\{\lambda \in \C: \Re \lambda\geq 0\}$  (i.e., we can assume that $T$ is accretive). Therefore, we can apply \hyperref[corr_trasf_accr]{Theorem} \ref{corr_trasf_accr}: the operator
$$
\tau(T)=(T-I)(T+I)^{-1}
$$ 
with domain $D(\tau(T))=R(T+I)$ and range $R(\tau(T))=R(T-I)$ is a contraction, $I-\tau(T)$ is invertible and  $T=(I+\tau(T))(I-\tau(T))^{-1}$. \\
In general, $\tau(T)$ has an additional property, i.e., from
$$
\pin{(I+\tau(T))(I-\tau(T))^{-1}\xi}{\xi}=\pin{T\xi}{\xi}\in \rn, \qquad \forall \xi \in D(T), \n{\xi}=1 
$$
it follows that
\begin{equation*}
\pin{(I+\tau(T))\eta}{(I-\tau(T))\eta}\in \rn, \qquad \forall \eta \in R(T+I), \n{(I-\tau(T))\eta}=1.
\end{equation*}
Now, let $K$ be an operator on $\H$ such that $I-K$ is invertible and
\begin{equation}
\label{cond_tras_rn_2}
\pin{(I+K)\eta}{(I-K)\eta} \in \rn, \qquad \forall \eta \in D(K),\; \n{(I-K)\eta}=1.
\end{equation}
We note that $K$ is in particular a contraction since $\rn\sub \{\lambda\in \C:\Re \lambda \geq0\}$. Thus, we have that the operator $T=(I+K)(I-K)^{-1}$ with domain $D(T)=R(I-K)$ is well-defined, has numerical range in $\rn$ and $\tau(T)=K$. Hence, \hyperref[corr_trasf_accr]{Theorem} \ref{corr_trasf_accr} has the following result as particular case.

\begin{teor}
	\label{corr_trasf_rn}
	Let $\rn$ be a proper, convex subset of the half-plane $\{\lambda \in \C: \Re \lambda\geq 0\}$ of  $\C$. Then the transform $T\mapsto \tau(T)$ defines a one-to-one  correspondence, which preserves extensions, between all operators on $\H$ with numerical range in $\rn$ and all the operators $K$ on $\H$ such that $I-K$ are invertible and satisfying \emph{(\ref{cond_tras_rn_2})}.\\
	In particular, the transform $T\mapsto \tau(T)$ defines a one-to-one correspondence between all densely defined operators on $\H$ with numerical range in $\rn$ and all operators $K$ on $\H$ satisfying \emph{(\ref{cond_tras_rn_2})} and with $R(I-K)$ dense in $\H$.
\end{teor}

\begin{cor}
	\label{cor_max_rn_iso}
	An operator $T$ on $\H$ with numerical range $\rn$ is  $\rn$-maximal if and only if the operator $\tau(T)=(T-I)(T+I)^{-1}$ is maximal in the set of operators $K$ satisfying \emph{(\ref{cond_tras_rn_2})} and such that $I-K$ are invertible.
\end{cor}

\no	For particular subsets $\rn$ (\ref{cond_tras_rn_2}) can be simplified. First of all, let us note that 
$$ \Re \pin{(I+K)\eta}{(I-K)\eta}=\n{\eta}^2-\n{K\eta}^2 $$ 
$$ \Im \pin{(I+K)\eta}{(I-K)\eta} =\frac{1}{i}  (\pin{K\eta}{\eta}-\pin{\eta}{K\eta})=2\Im \pin{K\eta}{\eta},$$
for all $\eta \in D(K)$. Moreover, for a positive homogeneous subset $\rn$ (i.e., such that $\mu \rn = \rn$ for all $\mu >0$) condition (\ref{cond_tras_rn_2}) is equivalent to 
$$
\pin{(I+K)\eta}{(I-K)\eta} \in \rn, \qquad \forall \eta \in D(K).
$$
\begin{itemize}
	\item If $\rn=\{\lambda\in \C:\Re \lambda \geq 0, \Im \lambda \geq 0 \}$, then (\ref{cond_tras_rn_2})  holds if and only if
	$$
	\n{K\eta} \leq \n{\eta} \text{ and } \Im \pin{K\eta}{\eta}\geq 0 \text{ for all } \eta \in D(K),
	$$
	i.e., if and only if $K$ is a contraction with numerical range in the upper semi-plane of $\C$. 
	\item If $\alpha> 0$ and $\rn=\{\lambda\in \C:0\leq \Re \lambda \leq \alpha\}$, then (\ref{cond_tras_rn_2})  holds if and only if
	$$
	0\leq \n{\eta}^2-\n{K\eta}^2\leq \alpha
	$$ 
	for all  $\eta \in D(K),\n{(I-K)\eta}=1$. This condition is equivalent to 
	$$
	0\leq \n{\eta}^2-\n{K\eta}^2\leq \alpha \n{(I-K)\eta}^2  \text{ for all } \eta \in D(K).
	$$ 
	If, moreover, $\alpha=1$, then (\ref{cond_tras_rn_2})  holds if and only if
	$$
	0\leq \n{\eta}^2-\n{K\eta}^2\leq \n{(I-K)\eta}^2=\n{\eta}^2-2\Re\pin{K\eta}{\eta}+\n{K\eta}^2
	$$  
	for all $\eta \in D(K)$, i.e.,
	$$
	\n{K\eta} \leq \n{\eta} \text{ and } \Re\pin{K\eta}{\eta}\leq \n{K\eta}^2  \text{ for all } \eta \in D(K).
	$$
	\item If $\rn=\{\lambda\in \C:\Re \lambda =0 \}$, then (\ref{cond_tras_rn_2}) holds if and only if
	$\n{K\eta} = \n{\eta}$ for all $\eta \in D(K)$, i.e., if and only if $K$ is an isometry.\\
	This case is not surprising since we have, up to a rotation, exactly the Cayley transform of a symmetric operator (see \cite[Theorem 13.5]{Schm}).
	\item If $\rn=\{\lambda\in \C:\Re \lambda \geq 0, \Im \lambda =0\}$, then (\ref{cond_tras_rn_2}) holds if and only if
	$$
	\n{K\eta} \leq \n{\eta} \text{ and } \pin{K\eta}{\eta}=\pin{\eta}{K\eta} \text{ for all } \eta \in D(K),
	$$
	i.e., if and only if $K$ is a symmetric contraction.\\
	In this case, the correspondence of \hyperref[corr_trasf_rn]{Theorem} \ref{corr_trasf_rn} is that given by Proposition 13.22 of  \cite{Schm}, and the mapping $T\mapsto \tau(T)$ is called {\it Krein transform}.
	\item If $\rn$ is a sector $\rn=\{\lambda\in \C:|\arg(\lambda)|\leq \theta\}$, where $0<\theta<\frac{\pi}{2}$, then (\ref{cond_tras_rn_2})  holds if and only if $\n{\sin(\theta) K\eta  \pm i  \cos(\theta) \eta}\leq \n{\eta}$, for all $\eta\in D(K)$. \\
	In particular, if $D(K)=\H$ then (\ref{cond_tras_rn_2}) is equivalent to
	\begin{equation}
	\label{K_sect}
	\n{\sin(\theta) K \pm i \cos(\theta)I }\leq 1.
	\end{equation}
	The class $C(\theta)$ of operators $K\in \B(\H)$ satisfying (\ref{K_sect}) has been studied in \cite{Arlinski,Arlinski2,Kolmanovich}. It
	has been used in descriptions of maximal sectorial extensions of sectorial operators and in the study of one-parameter semigroups of contractions  $U(t)=\exp(-t T)$, $t\ge 0$, generated by maximal sectorial operators $T$. 
\end{itemize}

\section{Operators associated to solvable sesquilinear forms}
\label{sec:solv}	

In this section we deal with the $\rn$-maximality of operators associated to sesquilinear forms. In particular, we work with {\it solvable} forms, that have been studied in \cite{Tp_DB,RC_CT,Second}. For reader's convenience we recall some important notions and results about them.\\

\no 	We assume that $\D$ is a dense subspace of $\H$ and we denote by $\iota$ the sesquilinear form which corresponds to the inner product, i.e.,  $\iota(\xi,\eta)=\pin{\xi}{\eta}$, with $\xi,\eta \in \H$.	A sesquilinear form $\O$ on $\D$ is called {\it  q-closed with respect to} a norm on $\D$ denoted by $\noo$ if
\begin{enumerate}
	\item there exists $\alpha>0$ such that $\n{\xi}\leq \alpha \n{\xi}_\O$, for all $\xi \in \D$, i.e., the embedding  $\D[\noo]\to \H$ is continuous;
	\item $\D[\noo]$ is a reflexive Banach space;
	\item there exists $\beta >0$ such that $|\O(\xi,\eta)|\leq \beta\n{\xi}_\O\n{\eta}_\O$, for all  $\xi,\eta \in \D$, i.e., $\O$ is bounded on $\D[\noo]$.\\
\end{enumerate}

\no	Let $\O$ be a q-closed sesquilinear form on $\D$ with respect to a norm $\noo$ and $\Eo:=\D[\noo]$. Let $\Eo^\times$ be the conjugate dual of $\Eo$.	If the set $\Po(\O)$ of bounded sesquilinear forms $\Up$ on $\H$ satisfying
\begin{enumerate}
	\item if $(\O+\Up)(\xi,\eta)=0$ for all $\eta\in \D$ then $\xi=0$;
	\item for all $\L\in \Eo^\times $ there exists $\xi\in \Eo$ such that the action of $\L$ on $\xi$ is given by
	$
	\pin{\L}{\eta}=(\O+\Up)(\xi,\eta)$,  for all $\eta \in \Eo,
	$
\end{enumerate}
is not empty, then $\O$ is said to be {\it solvable with respect to}  $\noo$.\\ 		

\no	The following result gives the representation theorem of solvable forms, whose first version is in \cite{Tp_DB}.

\begin{teor}[{\cite[Theorem 4.6]{RC_CT}, \cite[Theorem 2.7]{Second}}]
	\label{th_rapp_risol}
	Let $\O$ be a solvable sesquilinear form on $\D$ with respect to a norm $\noo$. Then there exists a closed operator $T$, with dense domain $D(T)\subseteq \D$ in $\H$, such that the following statements hold.
	\begin{enumerate}
		\item $\O(\xi,\eta)=\pin{T\xi}{\eta},$ for all $\xi\in D(T),\eta \in \D$.
		\item $D(T)$ is dense in $\D[\noo]$.
		\item A bounded form $\Up(\cdot,\cdot)=\pin{B\cdot}{\cdot}$ belongs to $\Po(\O)$ if and only if
		$0\in \rho(T+B)$. In particular, if $\Up=-\lambda \iota$, with $\lambda \in \C$, then $\Up\in \Po(\O)$ if and only if $\lambda \in \rho(T)$.
	\end{enumerate}
	The operator $T$ is uniquely determined by the following condition. Let $\xi,\chi\in \H$. Then $\xi\in D(T)$ and $T\xi=\chi$ if and only if $\xi\in \D$ and $\O(\xi,\eta)=\pin{\chi}{\eta}$ for all $\eta$ belonging to a dense subset of $\D[\noo]$.
\end{teor}

\no The operator $T$ in  \hyperref[th_rapp_risol]{Theorem} \ref{th_rapp_risol} is called {\it associated} to $\O$.

\begin{pro}[{\cite[Proposition 4.13]{RC_CT}}]
	The numerical range of the operator associated to a solvable sesquilinear form is a dense subset of the numerical range of the form.
\end{pro}

\no	Many sesquilinear forms studied in the literature are solvable (we refer to Section 7 of \cite{RC_CT}). In particular, the forms considered by Kato \cite[Theorem VI.2.1]{Kato} and McIntosh \cite[Theorem 3.1]{McIntosh68} are solvable (see \cite[Example 5.8]{Tp_DB} and \cite[Theorem 7.8]{RC_CT}).\\  Kato and McIntosh's theorems establish also that the associated operators are maximal sectorial and maximal accretive, respectively.
Hence, a natural question arises: is the operator associated to a solvable form with numerical range contained in $\rn$ (different from $\C$) $\rn$-maximal? 
By \cite[Corollary 4.14]{RC_CT}, the operators associated to symmetric solvable forms are self-adjoint, then, in particular, maximal symmetric. In the following we formulate other results on maximality of the associated operators.

\begin{teor}
	\label{th_op_ass_max}
	Let $\rn$ be a proper, convex subset of $\C$ and let $\O$ be a solvable sequilinear form on $\D$, with numerical range $\rn_\O\sub \rn$ and associated operator $T$. If a sesquilinear form $\Up\in \Po(\O)$ has numerical range $\rn_\Up$ such that $\rn \cap (-\rn_\Up)=\varnothing$, then $T$ is $\rn$-maximal.		
	In particular, if there exists $\lambda\in \rn^c$ such that $-\lambda \iota\in \mathfrak{P}(\O)$, then $T$ is $\rn$-maximal.
\end{teor}
\begin{proof}
	The numerical range $\rn_T$ of $T$ is contained in $\rn$. Let $B$ the  operator associated to $\Up$ and $\rn_B$ be the numerical range of $B$. By \hyperref[th_rapp_risol]{Theorem} \ref{th_rapp_risol}, $T+B$ is a bijection. 
	Let $T'$ be an extension of $T$ with numerical range contained in $\rn$. Thus $T'+B$ is injective, because $\rn_{T'} \subseteq \rn$, $\rn_B=\rn_\Up$ and $\rn \cap (-\rn_\Up)=\varnothing$. 
	Consequently $T=T'$, i.e., $T$ is $\rn$-maximal.
\end{proof}

\begin{cor}
	\label{cor_form_n_max}
	Let $\rn$ be a proper, convex subset of $\C$ and let $\O$ be a q-closed sequilinear form on $\D$, with numerical range $\rn_\O\sub \rn$. Assume that one of the following statements holds.
	\begin{enumerate}
		\item[(i)] If $\{\xi_n\}$ is a sequence in $\D$ such that $\displaystyle \lim_{n\to \infty }\n{\xi_n}= 0$ and $\displaystyle\lim_{n\to \infty} |\O(\xi_n,\xi_n)|= 0$, then $\displaystyle\lim_{n\to \infty} \n{\xi_n}_\O=0$. 
		\item[(ii)] There exists a bounded form $\Up$ on $\H$ such that $\rn_\O\cap (-\rn_{\Up})=\varnothing$, where $\rn_{\Up}$ is the numerical range of $\Up$, and (ii') or (ii'') below holds
		\begin{enumerate}
			\item[(ii')] if $\{\xi_n\}$ is a sequence in $\D$ such that $\displaystyle \sup_{\n{\eta}_\O=1} |(\O+\Up)(\xi_n,\eta)|\to 0$, then $ \n{\xi_n}_\O\to 0$;
			\item[(ii'')] there exists a constant $c>0$ such that
			$$ c\n{\xi}_\O\leq \sup_{\n{\eta}_\O=1} |(\O+\Up)(\xi,\eta)|, \qquad \forall \xi \in \D.$$
		\end{enumerate}
	\end{enumerate}
	Then $\O$ is solvable and its associated operator $T$ is $\rn$-maximal.
\end{cor}
\begin{proof}
	This is an application of \cite[Theorem 5.2, Corollary 5.3, Theorem 5.4]{RC_CT} and of \hyperref[th_op_ass_max]{Theorem} \ref{th_op_ass_max}.
\end{proof}

\no	The $\rn$-maximality of operators associated to solvable forms holds in any case if the numerical range of the form is contained in a strip.

\begin{teor}
	\label{th_strip_op_assoc}
	Let $\O$ be a solvable sesquilinear form on $\D$ with respect to a norm $\noo$ and with numerical range $\rn_\O$ contained in a strip $\St$. Let $T$ be its associated operator with numerical range $\rn_T$. Then $\ol{\rn_T}^c\subseteq \rho(T)$ and $T$ is $\rn$-maximal, where $\rn$ is any proper, convex subset of $\C$ containing $\rn_T$.
\end{teor}
\begin{proof}
	We can assume again, without loss of generality, that $\ol{\St}=\ol{\St_\alpha}:=\{\lambda \in \C:|\Im \lambda|\leq \alpha\}$, for some $\alpha\geq 0$.\\
	Consider the real and imaginary parts $\Re \O$, $\Im \O$ of $\O$. Since 
	\begin{equation}
	\label{O=Re O+iIm O}
	\O(\xi,\xi)=\Re \O(\xi,\xi)+i\Im \O(\xi,\xi), \qquad \forall \xi \in \D,
	\end{equation}
	and $\rn_\O\subseteq \St$, then  $\Im \O$  has numerical range in $[-\alpha,\alpha]$, so it is bounded and it extends to a bounded sesquilinear form $\Psi$ on $\H$. Moreover, $\Re \O$ is solvable with respect to $\noo$, being a difference of a solvable form and a bounded form. \\
	Let $S$ be the operator associated to $\Re \O$ and $B$ be the bounded operator such that $\Psi(\xi,\eta)=\pin{B\xi}{\eta}$, for all $\xi,\eta\in \H$. From (\ref{O=Re O+iIm O}) it follows that $S+iB$ is exactly the operator associated to $\O$, i.e., $T=S+iB$. But $S$ is self-adjoint by \cite[Corollary 4.14]{RC_CT}, and $B$ is, too. Therefore, $T=S+iB$ is the decomposition of \hyperref[lem_decomp_strip]{Lemma} \ref{lem_decomp_strip}. With the same argument of the resolvent set under perturbation used in \hyperref[th_car_strip]{Theorem} \ref{th_car_strip}, $\ol{\rn_T}^c\subseteq \rho(T)$, and the rest of the statement follows by \hyperref[pro_rn_max_dens]{Proposition} \ref{pro_rn_max_dens}.
\end{proof}

\no	We recall that (\cite[Definition 4.1]{Second}) a solvable sesquilinear form $\O$ on $\D$ with associated operator $T$ is said {\it hyper-solvable} if $\D=D(|T|^\mez)$. Under this condition one has the following Kato's second type representation (see \cite[Theorem VI.2.23]{Kato} and \cite[Theorem 4.17]{Second})
$$
\O(\xi,\eta)=\pin{U|T|^\mez \xi}{|T^*|^\mez \eta}, \qquad \forall \xi,\eta \in \D,
$$
where $T=U|T|=|T^*|U$ is the polar decomposition of $T$.\\
For hyper-solvable sesquilinear forms the converse of \hyperref[th_strip_op_assoc]{Theorem} \ref{th_strip_op_assoc} holds as follows.

\begin{pro}
	Let $T$ be a densely defined $\rn$-maximal operator, where $\rn$ is contained in a strip $\St$, and in particular $\ol{\rn}^c\subseteq \rho(T)$. Then there exists a unique hyper-solvable sesquilinear form with associated operator $T$.		
\end{pro}
\begin{proof}
	It is a direct consequence of \hyperref[cor_strip_auto]{Corollary} \ref{cor_strip_auto} and \cite[Theorem 5.1]{Second}.	
\end{proof}

\no Moreover, next result simplifies the criterion of Lemma 4.14 of \cite{Second} when the numerical range of the form is contained in a strip (see also Corollary 4.16 of \cite{Second}). 

\begin{cor}
	If $\O$ is a solvable sesquilinear form on $\D$ with respect to an inner product and with associated operator $T$. If the numerical range $\rn_\O$ of $\O$ is contained in a strip, then the following statements are equivalent.
	\begin{enumerate}
		\item $\D=D(|T|^\mez)$, i.e., $\O$ is hyper-solvable;
		\item $\D\sub D(|T|^\mez)$;
		\item $\D\supseteq D(|T|^\mez)$.
	\end{enumerate}
\end{cor}
\begin{proof}
	By \hyperref[th_strip_op_assoc]{Theorem} \ref{th_strip_op_assoc} and \hyperref[cor_strip_auto]{Corollary}  \ref{cor_strip_auto}, $D(T)=D(T^*)$. Hence \cite[Corollary 1.3]{Curgus} implies that $D(|T|^\mez)=D(|T^*|^\mez)$. Therefore we conclude with Lemma 4.14 of \cite{Second}.
\end{proof}

\no Finally, by \hyperref[th_car_n_max_set]{Theorem} \ref{th_car_n_max_set}, it is also possible to make more precise the correspondence, given by \cite[Theorem VI.2.6]{Kato}, between densely defined, closed, sectorial forms and m-sectorial operators as follows.

\begin{cor}
	\label{corr_form_oper_n_max}
	Let $\rn\subset \C$  be a closed, convex subset of a sector of $\C$. There exists a one-to-one correspondence between all closed, densely defined sesquilinear forms with numerical range in $\rn$ and all $\rn$-maximal, densely defined operators on $\H$.
\end{cor}

\subsection*{Acknowledgments}
This work was supported by the project ''Problemi spettrali e di rappresentazione in quasi *-algebre di operatori'', 2017, of the ''National Group for Mathematical Analysis, Probability and their Applications'' (GNAMPA – INdAM).

\vspace*{0.5cm}
\begin{center}
	\textsc{Rosario Corso, Dipartimento di Matematica e Informatica} \\
	\textsc{Università degli Studi di Palermo, I-90123 Palermo, Italy} \\
	{\it E-mail address}: {\bf rosario.corso@studium.unict.it}
\end{center}

\end{document}